\documentclass[a4paper,12pt]{amsart}

\usepackage{amsmath}
\usepackage{amsthm}
\usepackage{amssymb}
\usepackage{amsfonts}
\usepackage{enumitem} 
\usepackage{color}
\usepackage[bookmarks=true,hyperindex,pdftex,colorlinks,citecolor=blue]{hyperref}
\usepackage{geometry}
\usepackage{marginnote}

\DeclareMathOperator{\Lip}{Lip}     
\DeclareMathOperator{\lip}{lip}     
\DeclareMathOperator{\ext}{ext}     
\DeclareMathOperator{\diam}{diam}   
\DeclareMathOperator{\lspan}{span}  
\DeclareMathOperator{\supp}{supp}   

\newcommand{\abs}[1]{\left|{#1}\right|}                     
\newcommand{\pare}[1]{\left({#1}\right)}                    
\newcommand{\set}[1]{\left\{{#1}\right\}}                   
\newcommand{\norm}[1]{\left\|{#1}\right\|}                  
\newcommand{\dual}[1]{{#1}^\ast}                            
\newcommand{\ddual}[1]{{#1}^{\ast\ast}}                     
\newcommand{\ball}[1]{B_{{#1}}}                             
\newcommand{\duality}[1]{\left<{#1}\right>}                 
\newcommand{\cl}[1]{\overline{#1}}                          
\newcommand{\wscl}[1]{\overline{#1}^{\dual{w}}}             
\newcommand{\weaks}{\textit{w}$^\ast$}                      
\newcommand{\wsconv}{\stackrel{\dual{w}}{\longrightarrow}}  

\newcommand{\lipfree}[1]{\mathcal{F}({#1})}  
\newcommand{\lipnorm}[1]{\norm{#1}_L}        
\newcommand{\mol}[1]{u_{#1}}                 
\newcommand{\embd}{\delta}                   
\newcommand{\slack}[1]{\varepsilon({#1})}    
\newcommand{\ideal}[1]{\mathcal{I}({#1})}    
\newcommand{\hull}[1]{\mathcal{H}({#1})}     
\newcommand{\setd}[1]{\mathcal{D}_{#1}}      
\newcommand{\sete}[1]{\mathcal{E}_{#1}}      

\theoremstyle{plain}
\newtheorem{theorem}{Theorem}[section]
\newtheorem{lemma}[theorem]{Lemma}
\newtheorem{corollary}[theorem]{Corollary}
\newtheorem{proposition}[theorem]{Proposition}
\newtheorem*{claim}{Claim}

\theoremstyle{definition}
\newtheorem*{definition*}{Definition}
\newtheorem{definition}[theorem]{Definition}
\newtheorem{example}[theorem]{Example}
\newtheorem{question}{Question}

\theoremstyle{remark}

\newtheorem*{remark*}{Remark}

\begin{document}

\title{Supports and extreme points in Lipschitz-free spaces}

\author[R. J. Aliaga]{Ram\'on J. Aliaga}
\address[R. J. Aliaga]{Instituto Universitario de Matem\'atica Pura y Aplicada, Universitat Polit\`ecnica de Val\`encia, Camino de Vera S/N, 46022 Valencia, Spain}
\email{raalva@upvnet.upv.es}

\author[E. Perneck\'a]{Eva Perneck\'a}
\address[E. Perneck\'a]{Faculty of Information Technology, Czech Technical University in Prague, Th\'akurova 9, 160 00, Prague 6, Czech Republic}
\email{perneeva@fit.cvut.cz}



\begin{abstract}
For a complete metric space $M$, we prove that the finitely supported extreme points of the unit ball of the Lipschitz-free space $\lipfree{M}$ are precisely the elementary molecules $(\embd(p)-\embd(q))/d(p,q)$ defined by pairs of points $p,q$ in $M$ such that the triangle inequality $d(p,q)<d(p,r)+d(q,r)$ is strict for any $r\in M$ different from $p$ and $q$. To this end, we show that the class of Lipschitz-free spaces over closed subsets of $M$ is closed under arbitrary intersections when $M$ has finite diameter, and that this allows a natural definition of the support of elements of $\lipfree{M}$.
\end{abstract}


\subjclass[2010]{Primary 46B20; Secondary 54E50}

\keywords{Lipschitz-free space, extreme point, support}

\maketitle


\section{Introduction}

Let $\Lip_0(M)$ denote the space of real-valued Lipschitz functions on a pointed metric space $(M,d)$ (i.e. one with a designated base point) that vanish at the base point, endowed with the Lipschitz norm
$$
\lipnorm{f}=\sup\set{\frac{f(p)-f(q)}{d(p,q)}:p\neq q\in M}.
$$Then $\Lip_0(M)$ is a dual Banach space and the map $\embd$ that assigns to each $x\in M$ its evaluation functional $\embd(x)\colon f\mapsto f(x)$ embeds $M$ isometrically into $\dual{\Lip_0(M)}$. Moreover, these functionals span a space $\lipfree{M}=\cl{\lspan}\,\embd(M)$ that can be canonically identified with a predual of $\Lip_0(M)$. The spaces $\lipfree{M}$ were prominently featured as ``Arens-Eells spaces'' in the authoritative monograph \cite{Weaver} due to Weaver. Later, the study of their applications in nonlinear geometry of Banach spaces was initiated in \cite{GoKa_2003} by Godefroy and Kalton, who also introduced the name \mbox{\emph{Lipschitz-free spaces}} based on their universal property. We refer to \cite{Weaver} for basic facts on Lipschitz and Lipschitz-free spaces, and to the survey \cite{Godefroy_2015} and the references therein for more recent progress in understanding their Banach space properties.

The linear structure of Lipschitz-free spaces is not straightforward to analyze and has been the subject of vigorous recent research efforts. In particular, the extremal structure of their unit ball has not yet been completely described. The first important step in this direction was Weaver's proof that any preserved extreme point of $\ball{\lipfree{M}}$ must be an \emph{elementary molecule} \cite{Weaver_1995}, that is, an element of the form
$$
\mol{pq}:=\frac{\embd(p)-\embd(q)}{d(p,q)}
$$
for some $p\neq q\in M$; note that $\norm{\mol{pq}}=1$. This allows us to restate the problem of characterizing certain types of extreme points as finding equivalent geometric conditions on pairs of points $p,q$ in $M$. It is easy to see that one such necessary condition is that the \emph{metric segment}
$$
[p,q]:=\set{x\in M:d(p,q)=d(p,x)+d(q,x)}
$$
consists only of the points $p$ and $q$.

Progress in this direction was mostly stalled until very recently. In \cite{GaPrRu_2018}, Garc\'ia-Lirola, Proch\'azka and Rueda Zoca gave a complete geometric characterization of the strongly exposed points of $\ball{\lipfree{M}}$ (see Theorem \ref{tm:prext_strexp}\ref{tm:prext_strexp:strexp}). In \cite{AlGu_2019}, the first author and Guirao gave a similar geometric characterization of preserved extreme points (see Theorem \ref{tm:prext_strexp}\ref{tm:prext_strexp:prext}), and asked whether extreme points could be described analogously. In particular, they asked if it is true that $u_{pq}$ is extreme if and only if $[p,q]=\set{p,q}$ (see Question 1 in \cite{AlGu_2019}). The answer is positive when $M$ is compact by Theorem 4.2 in \cite{AlGu_2019}. Concurrently, Garc\'ia-Lirola, Petitjean, Proch\'azka and Rueda Zoca proved in \cite{GPPR_2018} that all preserved extreme points of $\ball{\lipfree{M}}$ are denting points, and gave a positive answer to Question 1 in \cite{AlGu_2019} for bounded, uniformly discrete $M$.

Our main result in this note is that the answer to Question 1 in \cite{AlGu_2019} is positive for any complete metric space $M$, that is:

\begin{theorem}
\label{tm:main_theorem}
Let $M$ be a complete pointed metric space and let $\mu\in\lipfree{M}$ be finitely supported, i.e. $\mu\in\lspan\,\embd(M)$. Then $\mu$ is an extreme point of $\ball{\lipfree{M}}$ if and only if $\mu=\mol{pq}$ for distinct points $p,q$ in $M$ such that $[p,q]=\set{p,q}$.
\end{theorem}

\noindent The proof relies on a refinement of the methods used in \cite{AlGu_2019} for obtaining the characterization of preserved extreme points, but an additional key observation is required: the fact that the class of Lipschitz-free spaces over closed subsets of $M$, considered as subspaces of $\lipfree{M}$, is closed under intersections when $M$ is bounded (see Theorem \ref{tm:f_intersection}). We prove this by considering the algebra structure of \mbox{$\Lip_0(M)$ --- we} show that the \weaks-closure of any ideal in $\Lip_0(M)$ is also an ideal and appeal to results in Chapter 4 of \cite{Weaver}. As another consequence of this fact we show that, for a bounded $M$, the support of  any element of $\lipfree{M}$ can be defined in a natural way.

\section{Preliminaries}

We will now briefly describe the notation used throughout the text. $M$ will be a complete pointed metric space with metric $d$. We will write $\diam(M)$ for the (possibly infinite) diameter of $M$, and denote
$$
\slack{x;p,q}:=d(p,x)+d(q,x)-d(p,q)
$$
for $p,q,x\in M$. Note that $\slack{x;p,q}\geq 0$, and $\slack{x;p,q}=0$ if and only if $x\in[p,q]$. We will also consider the Stone-\v{C}ech compactification $\beta M$ of $M$ and the fact that any real-valued continuous function on $M$ may be extended continuously and uniquely to a function on $\beta M$, possibly by adding $\pm\infty$ to its range. In particular, for $\xi\in\beta M$ and $p,q\in M$, $d(\xi,p)$ and $\slack{\xi;p,q}$ are well-defined values in $[0,\infty]$. Moreover, by an argument given in the proof of Proposition 2.1.6 in \cite{Weaver}, $\inf_{p\in M}d(\xi,p)>0$ for any $\xi\in\beta M\setminus M$.

For any subset $N$ of $M$, it is well known that any element of $\Lip_0(N)$ may be extended to an element of $\Lip_0(M)$ without increasing its Lipschitz norm or its supremum, and that therefore $\lipfree{N}$ can be identified with a subspace of $\lipfree{M}$. Namely, there exists a linear isometric embedding $\hat{\iota}:\,\lipfree{N}\to \lipfree{M}$, induced by the identity mapping $i:\,N\to M$, such that $\hat{\iota}(\embd_N(x))=\embd_M(i(x))$ for every $x\in N$; here $\embd_N$ and $\embd_M$ are the embeddings of the respective metric spaces into the corresponding Lipschitz-free spaces. Hence, $\lipfree{N}$ is linearly isometric to the closed subspace of $\lipfree{M}$ spanned by the set $\embd_M(N)$ (see e.g. p. 91 in \cite{Godefroy_2015}), and we shall frequently identify them.

We will also use the following known fact about representation of elements of $\lipfree{M}$. We include a proof for the sake of completeness.

\begin{lemma}
\label{lm:l1_series_argument}
Let $M$ be a pointed metric space and $\mu\in\lipfree{M}$. For every $\varepsilon>0$, there exist sequences $(a_n)$ in $\mathbb{R}$ and $(p_n),(q_n)$ in $M$, with $p_n\neq q_n$ for all $n\in\mathbb{N}$, such that $\mu=\sum_{n=1}^\infty a_n\mol{p_nq_n}$ and $\sum_{n=1}^\infty\abs{a_n}<\norm{\mu}+\varepsilon$.
\end{lemma}

\begin{proof}
Let $\mu\in\lipfree{M}$ and $\varepsilon>0$. By Lemma 3.100 in \cite{FHHMZ_2011}, there exists a sequence $(\mu_n)\subset\lspan\,\embd(M)$ such that $\mu=\sum_{n=1}^\infty\mu_n$ and 
$$\sum_{n=1}^\infty\|\mu_n\| < \|\mu\|+\frac{\varepsilon}{2}.$$ Since the $\lipfree{M}$-norm on $\lspan\,\embd(M)$ can be computed by the formula 
$$\|\nu\|=\inf\left\{\sum_{i=1}^I|a_i|\colon \nu=\sum_{i=1}^Ia_i\frac{\delta(p_i)-\delta(q_i)}{d(p_i,q_i)}, I\in\mathbb{N}, a_i\in\mathbb{R}, p_i, q_i\in M, p_i\neq q_i\right\}$$
for each $\nu\in\lspan\,\embd(M)$ (see Section 2 in \cite{CDW_2016}), for each $n\in\mathbb{N}$ we can find a representation 
$$\mu_n=\sum_{i=1}^{I_n}a^n_i\frac{\delta(p^n_i)-\delta(q^n_i)}{d(p^n_i,q^n_i)}$$ such that $$\sum_{i=1}^{I_n}|a_i^n|<\|\mu_n\|+\frac{\varepsilon}{2^{n+1}}.$$
We re-index the sequences $(a_i^n)_{n,i}$, $(p_i^n)_{n,i}$ and $(q_i^n)_{n,i}$ as $(a_j)_{j=1}^\infty$, $(p_j)_{j=1}^\infty$ and $(q_j)_{j=1}^\infty$, respectively. Then 
\begin{align*}
\sum_{j=1}^{\infty}|a_j|=\sum_{n,i}|a_i^n|&=\sum_{n=1}^\infty\sum_{i=1}^{I_n}|a_i^n|\leq\sum_{n=1}^\infty\left(\|\mu_n\|+\frac{\varepsilon}{2^{n+1}}\right) < \|\mu\|+\varepsilon.
\end{align*}
Hence $(a_j)\in\ell_1$ and the series $\sum_{j=1}^\infty a_j \frac{\delta(p_j)-\delta(q_j)}{d(p_j,q_j)}$ converges absolutely in $\lipfree{M}$. Moreover, 
$$\mu=\sum_{j=1}^\infty a_j \frac{\delta(p_j)-\delta(q_j)}{d(p_j,q_j)}.$$ Indeed, for any $\xi>0$ find $N\in\mathbb{N}$ such that for every $m\geq N$ we have that $\max\left\{\left\|\mu-\sum_{j=1}^m\mu_j\right\|,\|\mu_m\|,\frac{\varepsilon}{2^m}\right\}<\xi$. Then for every $n>\sum_{k=1}^NI_k$ we also get
\begin{align*}
\left\|\mu-\sum_{j=1}^n a_j \frac{\delta(p_j)-\delta(q_j)}{d(p_j,q_j)}\right\|&\leq\left\|\mu-\sum_{j=1}^{m-1} \mu_j\right\|+\sum_{i=1}^{I_m}|a_i^m|\\
&\leq\left\|\mu-\sum_{j=1}^{m-1} \mu_j\right\|+\left\|\mu_m\right\|+\frac{\varepsilon}{2^m}<3\xi,
\end{align*}
where $m\in\mathbb{N}$ satisfies $\sum_{k=1}^{m-1}I_k<n\leq\sum_{k=1}^m I_k$.
\end{proof}

Given a Banach space $X$, its closed unit ball will be denoted by $\ball{X}$, and the evaluation of a functional $\dual{x}\in\dual{X}$ at $x\in X$ by $\duality{x,\dual{x}}$. We will consider the following types of extremal elements of $\ball{X}$: a point $x\in\ball{X}$ is
\begin{enumerate}[label={\upshape{(\roman*)}}]
\item an \emph{extreme point} of $\ball{X}$ if there are no $y,z\in\ball{X}\setminus\set{x}$ such that $x=\frac{1}{2}(y+z)$.
\item an \emph{exposed point} of $\ball{X}$ if there exists $\dual{x}\in\dual{X}$ such that $\norm{\dual{x}}=1$ and $x$ is the only element of $\ball{X}$ such that $\duality{x,\dual{x}}=1$.
\item a \emph{preserved extreme point} (also called \emph{\weaks-extreme point}) of $\ball{X}$ if it is an extreme point of $\ball{\ddual{X}}$.
\item a \emph{denting point} of $\ball{X}$ if there are slices of $\ball{X}$ (that is, sets of the form $\set{y\in\ball{X}:\duality{y,\dual{x}}>\alpha}$ for some $\dual{x}\in\dual{X}$ and $\alpha\in\mathbb{R}$) of arbitrarily small diameter containing $x$.
\item a \emph{strongly exposed point} of $\ball{X}$ if there exists $\dual{x}\in\dual{X}$ such that $\duality{x,\dual{x}}=1$ and that $\diam\set{y\in\ball{X}:\duality{y,\dual{x}}>1-\delta}\rightarrow 0$ when $\delta\rightarrow 0$.
\end{enumerate}
All of these elements have norm 1, and the implications (v)$\Rightarrow$(iv)$\Rightarrow$(iii)$\Rightarrow$(i) and (v)$\Rightarrow$(ii)$\Rightarrow$(i) hold. For further reference see e.g. \cite{GuMoZi_2014}. We will denote the set of extreme points of $\ball{X}$ as $\ext\ball{X}$.

Note also that the above concepts are invariant with respect to linear isometries and that a change of the base point in $M$ induces a linear isometry between the corresponding Lipschitz and Lipschitz-free spaces which preserves the elementary molecules. Therefore we may (and will) adapt the base point of $M$ without loss of generality.

The following statement describes the geometric characterizations of preserved extreme, denting, and strongly exposed points of $\ball{\lipfree{M}}$ in terms of the geometry of $M$. It summarizes Theorem 4.1 in \cite{AlGu_2019}, Theorem 2.4 in \cite{GPPR_2018} and Theorem 5.4 in \cite{GaPrRu_2018}:

\begin{theorem}
\label{tm:prext_strexp}
Let $M$ be a complete pointed metric space and $p,q$ be distinct points of $M$. Then:
\begin{enumerate}[label={\upshape{(\alph*)}}]
\item \label{tm:prext_strexp:prext} $\mol{pq}$ is a preserved extreme point of $\ball{\lipfree{M}}$ if and only if it is a denting point of $\ball{\lipfree{M}}$ if and only if $\slack{\xi;p,q}>0$ for all $\xi\in\beta M\setminus\set{p,q}$.
\item \label{tm:prext_strexp:strexp} $\mol{pq}$ is a strongly exposed point of $\ball{\lipfree{M}}$ if and only if there is a $C>0$ such that
\begin{equation}
\label{eq:strexp}
\min\set{d(x,p),d(x,q)}\leq C\cdot\slack{x;p,q}
\end{equation}
for all $x\in M$.
\end{enumerate}
\end{theorem}

There are several settings where one may apply the above criteria to show that $\mol{pq}$ must be a preserved extreme point or even a strongly exposed point whenever \mbox{$[p,q]=\set{p,q}$,} which are easy consequences or variations of known results but are not stated explicitly anywhere to the best of our knowledge. We collect them in \mbox{Proposition \ref{pr:state_of_the_art}.} Recall that a metric space $M$ has the \emph{Heine-Borel property} (we also say that $M$ is \emph{proper}) if the closed balls in $M$ are compact, and that $M$ is \emph{ultrametric} if \mbox{$d(x,y)\leq\max\set{d(x,z),d(y,z)}$} for any $x,y,z\in M$. Recall also the following definition from \cite{GPPR_2018}: a predual of $\lipfree{M}$ is \emph{natural} if it induces a \weaks-topology such that $\embd(M)\cap n\ball{\lipfree{M}}$ is \weaks-closed for any $n\in\mathbb{N}$.

\begin{proposition}
\label{pr:state_of_the_art}
Let $M$ be a complete pointed metric space, and let $p,q$ be distinct points of $M$ such that $[p,q]=\set{p,q}$. Then $\mol{pq}$ is a preserved extreme point of $\ball{\lipfree{M}}$ in these cases:

\begin{enumerate}[label={\upshape{(\alph*)}}]
\item \label{case:e2_proper} $M$ has the Heine-Borel property,
\item \label{case:e2_nat_predual} $\lipfree{M}$ has a natural predual,
\end{enumerate}

\noindent and it is a strongly exposed point of $\ball{\lipfree{M}}$ in the following cases:

\begin{enumerate}[label={\upshape{(\alph*)}}]
\addtocounter{enumi}{2}
\item \label{case:e2_ultrametric} $M$ is ultrametric,
\item \label{case:e2_tree} $\lipfree{M}$ is linearly isometric to $\ell_1(\Gamma)$ for some $\Gamma$.
\end{enumerate}
\end{proposition}

\begin{proof}
\ref{case:e2_proper} This is an easy extension of the compact case that was proved in Theorem 4.2 in \cite{AlGu_2019}. Indeed, suppose that $\mol{pq}$ is not a preserved extreme point of $\ball{\lipfree{M}}$, then by Theorem \ref{tm:prext_strexp}\ref{tm:prext_strexp:prext} there is $\xi\in\beta M\setminus\set{p,q}$ such that $d(p,\xi)+d(q,\xi)=d(p,q)$. Let $(x_i)$ be a net in $M$ that converges to $\xi$. Since $d(p,x_i)+d(q,x_i)$ converges to $d(p,q)$, we may assume that $(x_i)$ is bounded. The Heine-Borel property then implies that $(x_i)$ has a cluster point $x\in M$ that is different from $p,q$ and clearly $d(p,x)+d(q,x)=d(p,q)$.

\ref{case:e2_nat_predual} As in case \ref{case:e2_proper}, if $\mol{pq}$ is not preserved extreme then there is a bounded net $(x_i)$ in $M$ that converges to $\xi\in\beta M\setminus\set{p,q}$ and such that $d(p,x_i)+d(q,x_i)$ converges to $d(p,q)$. Then $(\embd(x_i))$ is a bounded net in $\lipfree{M}$ and so we may replace it by a \weaks-convergent subnet. Since $\lipfree{M}$ has a natural predual, there is $x\in M$ such that $\embd(x_i)\wsconv\embd(x)$. By \weaks-lower semicontinuity of the norm of $\lipfree{M}$ we get
\begin{align*}
d(p,x) = \norm{\embd(p)-\embd(x)} &\leq \liminf_i\norm{\embd(p)-\embd(x_i)} \\
&= \liminf_i d(p,x_i)=d(p,\xi)
\end{align*}
and similarly $d(q,x)\leq d(q,\xi)$. But then
$$
d(p,q)\leq d(p,x)+d(q,x)\leq d(p,\xi)+d(q,\xi)=d(p,q),
$$
so all inequalities are in fact equalities. In particular, $d(p,x)=d(p,\xi)>0$ and  $d(q,x)=d(q,\xi)>0$, thus $x\neq p,q$ and $d(p,x)+d(q,x)=d(p,q)$.

\ref{case:e2_ultrametric} Let $x\in M\setminus\set{p,q}$ and recall the following general property of ultrametric spaces: if $d(x,p)\neq d(p,q)$, then $d(x,q)=\max\set{d(x,p),d(p,q)}$ (see Property 3.3 in \cite{Dalet_2015_2}). Now distinguish three cases:
\begin{itemize}
\item If $d(x,p)=d(p,q)$, then $\slack{x;p,q}=d(x,q)$.
\item If $d(x,p)<d(p,q)$, then $d(x,q)=d(p,q)$ and $\slack{x;p,q}=d(x,p)$.
\item If $d(x,p)>d(p,q)$, then $d(x,q)=d(x,p)$ and $\slack{x;p,q}=2d(x,p)-d(p,q)$, hence $\slack{x;p,q}>d(x,p)$.
\end{itemize}
In all cases \eqref{eq:strexp} is satisfied with $C=1$, so $\mol{pq}$ is a strongly exposed point of $\ball{\lipfree{M}}$ by Theorem \ref{tm:prext_strexp}\ref{tm:prext_strexp:strexp}.

We remark that in an ultrametric space, the condition $[p,q]=\set{p,q}$ is true for any pair of points $p,q$, so all elementary molecules are strongly exposed.

\ref{case:e2_tree} By Theorem 5 in \cite{DaKaPr_2016}, $M$ is a negligible subset of an $\mathbb{R}$-tree $T$ containing all branching points of $T$. Let $0$ denote the root of $T$ and assign it as the base point of $M$. Write $p\prec q$ if $q$ is a successor of $p$, i.e. if $p\in[0,q]$. Let $p,q\in M$ be such that $[p,q]\cap M=\set{p,q}$. If neither $p\prec q$ nor $q\prec p$, then there is a branching point $r=p\wedge q$, defined by $[0,r]=[0,p]\cap [0,q]$, such that $r\prec p,q$ and $d(p,q)=d(p,r)+d(r,q)$, hence $r\in[p,q]\cap M$, a contradiction. So assume that $p\prec q$ and let $x\in M\setminus\set{p,q}$. Distinguish three cases:
\begin{itemize}
\item If $q\prec x$, then $q\in[p,x]$ and so $\slack{x;p,q}=2d(q,x)$.
\item If $x\prec p$, then $p\in[x,q]$ and so $\slack{x;p,q}=2d(p,x)$.
\item Otherwise, let $r=p\wedge x$, then $r\in [x,p]\cap M$ because it is a branching point, and also $r\in [x,q]$, hence $\slack{x;p,q}=2d(p,x)$.
\end{itemize}
So \eqref{eq:strexp} is satisfied with $C=1/2$ and the conclusion follows by Theorem \ref{tm:prext_strexp}\ref{tm:prext_strexp:strexp}.
\end{proof}

\section{Intersections of free spaces}
\label{section: intersections}

Recall that $\Lip_0(M)$ is an algebra under pointwise multiplication if (and only if) $M$ is bounded. Indeed, for any $f,g\in\Lip_0(M)$ we have
\begin{align}
\label{eq:mult_norm_bound}
\lipnorm{fg} &\leq \lipnorm{f}\norm{g}_\infty + \lipnorm{g}\norm{f}_\infty \\
\notag &\leq \lipnorm{f}\lipnorm{g}\diam(M) + \lipnorm{g}\lipnorm{f}\diam(M) \\
\notag &= 2\diam(M)\lipnorm{f}\lipnorm{g}
\end{align}
and so $fg\in\Lip_0(M)$. $\Lip_0(M)$ is not in general a Banach algebra, as that would require $\lipnorm{fg}\leq\lipnorm{f}\lipnorm{g}$, but it can be equivalently renormed to become one by dilating $M$ by a constant factor so that its diameter is less than $1/2$.

An ideal in $\Lip_0(M)$ is a subspace $Y$ (not necessarily closed) such that $fg\in Y$ for any $f\in Y$ and $g\in\Lip_0(M)$. Following Chapter 4 in \cite{Weaver}, for any set $K\subset M$ that contains the base point let us define
$$
\ideal{K}=\set{f\in\Lip_0(M):f(x)=0\text{ for all }x\in K},
$$
which is a \weaks-closed ideal of $\Lip_0(M)$. Note that we have $\lipfree{K}^\perp=\ideal{K}$ and $\ideal{K}_\perp=\lipfree{K}$ (see e.g. \cite{HaNo_2017}, where $\ideal{K}$ is denoted $\Lip_K(M)$). For any subspace $Y$ of $\Lip_0(M)$, let us also define the \emph{hull} of $Y$ as the closed set
$$
\hull{Y}=\set{x\in M:f(x)=0\text{ for all }f\in Y} \,.
$$
Notice that $\hull{\ideal{K}}=K$ for any closed $K\subset M$, as witnessed by the Lipschitz map $x\mapsto d(x,K):=\inf\set{d(x,y):y\in K}$.

We show next that an element of $\lipfree{M}$ endowed with a weight is again an element of $\lipfree{M}$.

\begin{lemma}
\label{lm:multiply_element}
Let $M$ be a bounded pointed metric space, $\mu\in\lipfree{M}$ and let \mbox{$g\in\Lip_0(M)$.} Define the function ${\mu\circ g}\colon\Lip_0(M)\rightarrow\mathbb{R}$ by
$$({\mu\circ g})(f)=\duality{\mu,fg}$$
for all $f\in\Lip_0(M)$. Then $\mu\circ g\in\lipfree{M}$.
\end{lemma}

\begin{proof}
It is clear that $\mu\circ g$ is a linear functional, and it follows immediately from \eqref{eq:mult_norm_bound} that
$$
\norm{\mu\circ g}\leq 2\diam(M)\norm{\mu}\lipnorm{g}.
$$
Hence $\mu\circ g\in\dual{\Lip_0(M)}$. Now let $(f_i)$ be a bounded net \weaks-converging to an $f$ in $\Lip_0(M)$. Since \weaks-convergence agrees with pointwise convergence in bounded subsets of $\Lip_0(M)$, it is easy to verify that then also $(f_ig)$ \weaks-converges to $fg$. Therefore 
$$
\lim_i\duality{f_i,\mu\circ g}=\lim_i\duality{\mu,f_i g}=\duality{\mu,fg}=\duality{f,\mu\circ g}.
$$
So, by the Banach-Dieudonn\'e theorem $\mu\circ g$ is \weaks-continuous and it belongs to $\lipfree{M}$.
\end{proof}

Using these ``weighted elements'' of $\lipfree{M}$, we can show that the \weaks-closure of any ideal in $\Lip_0(M)$ is again an ideal. Specifically:

\begin{proposition}
\label{pr:ideal_wscl}
Let $M$ be a bounded and complete pointed metric space and let \mbox{$Y\subset\Lip_0(M)$} be an ideal. Then $\wscl{Y}=\ideal{\hull{Y}}$.
\end{proposition}

\begin{proof}
Let $f\in\wscl{Y}$, $g\in\Lip_0(M)$ and $h=fg$, and let $U\subset\Lip_0(M)$ be a \mbox{\weaks-neighborhood} of $h$. Then $U$ contains a \weaks-neighborhood $V$ of the form
$$
V=\set{\psi\in\Lip_0(M): \abs{\duality{\mu_n,h-\psi}}<\varepsilon \text{ for } n=1,\ldots,N},
$$
where $\mu_n\in\lipfree{M}$, $\varepsilon>0$, and $N\in\mathbb{N}$. Consider the set
$$
W=\set{\phi\in\Lip_0(M): \abs{\duality{\mu_n\circ g,f-\phi}}<\varepsilon \text{ for } n=1,\ldots,N},
$$
where $\mu_n\circ g$ are as in Lemma \ref{lm:multiply_element}. Then $W$ is a \weaks-neighborhood of $f$, so there exists $\phi\in Y\cap W$. Let $\psi=\phi g$. Then $\psi\in Y$ since $Y$ is an ideal, and for any $n=1,\ldots,N$ we have
$$
\abs{\duality{\mu_n,h-\psi}}=\abs{\duality{\mu_n,(f-\phi)g}}=\abs{\duality{\mu_n\circ g,f-\phi}}<\varepsilon,
$$
so $\psi\in V$. Therefore $V\cap Y$ is nonempty, and it follows that $h\in\wscl{Y}$. We have thus proved that $\wscl{Y}$ is an ideal. By Corollary 4.2.6 in \cite{Weaver} we get \mbox{$\wscl{Y}=\ideal{\hull{\wscl{Y}}}$,} where completeness of $M$ is used. Clearly $\hull{\wscl{Y}}=\hull{Y}$, which ends the proof.
\end{proof}

We can now state the main results in this section. In what follows, the Lipschitz-free spaces over subsets of $M$ are identified with the corresponding subspaces of $\lipfree{M}$, as we have remarked before Lemma \ref{lm:l1_series_argument}.

\begin{theorem}
\label{tm:f_intersection}
Let $M$ be a bounded and complete pointed metric space, and let $\set{K_i:i\in I}$ be a family of closed subsets of $M$ containing the base point. Then
$$
\bigcap_{i\in I}\lipfree{K_i}=\mathcal{F}\pare{\bigcap_{i\in I}K_i} \,.
$$
\end{theorem}

\begin{proof}
We have
\begin{align*}
\bigcap_{i\in I}\lipfree{K_i} = \bigcap_{i\in I}\pare{\ideal{K_i}_\perp} &= \pare{\bigcup_{i\in I}\ideal{K_i}}_\perp \\
&= \pare{\lspan\set{\ideal{K_i}:i\in I}}_\perp \\
&= \pare{\wscl{\lspan}\set{\ideal{K_i}:i\in I}}_\perp \,.
\end{align*}
Since $\lspan\set{\ideal{K_i}:i\in I}$ is an ideal in $\Lip_0(M)$, we may apply Proposition \ref{pr:ideal_wscl} to get
$$
\bigcap_{i\in I}\lipfree{K_i} = \ideal{H}_\perp = \lipfree{H},
$$
where $H=\hull{\lspan\set{\ideal{K_i}:i\in I}}$. Now notice that $\bigcap_{i\in I}K_i\subset H$, and for each $x\notin\bigcap_{i\in I}K_i$ there exists $i\in I$ such that $x\notin K_i$, so the function $y\mapsto d(y,K_i)$ shows that $x\notin H$. Thus $H=\bigcap_{i\in I}K_i$ and this finishes the proof.
\end{proof}

Let us introduce now the notion of a support of an element of a free space pertinent to our context. 

\begin{definition} Consider a pointed metric space $M$ with the base point $0$. For a $\mu\in\lipfree{M}$ let the \emph{support} of $\mu$, denoted $\supp(\mu)$, be defined as the smallest closed set $K\subset M$ such that $\mu\in\lipfree{K\cup\set{0}}$, provided such a $K$ exists. That is, for any closed $L\subset M$ that contains the base point, $\mu\in\lipfree{L}$ if and only if $K\subset L$. 
\end{definition}

The conclusion of Theorem \ref{tm:f_intersection} can be equivalently restated in terms of supports in the following sense.

\begin{proposition}
\label{prop: intersections_iff_supports}
Let $M$ be a pointed metric space. The following are equivalent:
\begin{enumerate}[label={\upshape{(\roman*)}}]
\item If $\set{K_i:i\in I}$ is a family of closed subsets of $M$ that contain the base point, then $\bigcap_{i\in I}\lipfree{K_i}=\mathcal{F}\pare{\bigcap_{i\in I}K_i}$.
\item The support of $\mu$ exists for every $\mu\in\lipfree{M}$.
\end{enumerate}
\end{proposition}

\begin{proof}
Suppose that (i) holds and let $\mu\in\lipfree{M}$. Let $\mathcal{S}$ be the family of all closed sets $C\subset M$ such that $\mu\in\lipfree{C\cup\set{0}}$, and let $K=\bigcap_{C\in\mathcal{S}} C$. Then $\mu\in\bigcap_{C\in\mathcal{S}} \lipfree{C\cup\set{0}}=\lipfree{K\cup\set{0}}$ so $K\in\mathcal{S}$, and $K$ is clearly the smallest element of $\mathcal{S}$, so $K=\supp(\mu)$.

Now assume (ii) and let $\set{K_i:i\in I}$ be as in (i). Let $\mu\in\bigcap_{i\in I}\lipfree{K_i}$, then $\supp(\mu)\subset K_i$ for all $i$, thus $\supp(\mu)\subset\bigcap_{i\in I} K_i$ and $\mu\in\lipfree{\bigcap_{i\in I} K_i}$. Hence $\bigcap_{i\in I}\lipfree{K_i}\subset\lipfree{\bigcap_{i\in I} K_i}$; the reverse inclusion is trivial.
\end{proof}

We do not know whether $\supp(\mu)$ exists in general, but Theorem \ref{tm:f_intersection} and Proposition \ref{prop: intersections_iff_supports} give a class of metric spaces for which it does:

\begin{corollary}
\label{cr:bounded_support}
If $M$ is a bounded, complete pointed metric space, then the support of $\mu$ exists for every $\mu\in\lipfree{M}$.
\end{corollary}

Note that if $\mu\in\lspan\embd(M)$, that is, if $\mu=\sum_{i=1}^n a_i\embd(x_i)$ for $n\in\mathbb{N}$, $x_i\in M\setminus\set{0}$ and $a_i\in\mathbb{R}\setminus\set{0}$, then the support of $\mu$ is exactly the set $K=\set{x_1,\ldots,x_n}$. Indeed, clearly $\mu\in\lipfree{K\cup \{0\}}$, and if $L$ is a closed subset of $M$ containing the base point and such that $K\setminus L\neq\emptyset$, then we can choose $p\in K\setminus L$ and take $f\in\Lip_0(M)$ that vanishes on $L$ and $K\setminus\set{p}$ but satisfies $f(p)=1$ to show that $\mu\notin\lipfree{L}=\ideal{L}_\perp$. On the other hand, if the support of $\mu\in\lipfree{M}$ is some finite set $K=\set{x_1,\ldots,x_n}$, where $n\in\mathbb{N}$ and $x_i\in M\setminus\set{0}$, then $\mu=\sum_{i=1}^n a_i\embd(x_i)$ and all $a_i\in\mathbb{R}\setminus\set{0}$ because $K$ is the smallest such subset of $M$. Therefore, there shall be no ambiguity when speaking about finitely supported elements of $\lipfree{M}$.

\section{Extreme molecules}

We now proceed to our main result. Let $M$ be a complete pointed metric space and denote
$$
\widetilde M:=\set{(p,q)\in M^2: p\neq q}
$$
with the subspace topology of $M^2$. The \emph{de Leeuw transform} $\Phi$ assigns to a function $f:M\to\mathbb{R}$ the function $\Phi f:\widetilde M \to \mathbb{R}$ defined by
$$
\Phi f(p,q):=\frac{f(p)-f(q)}{d(p,q)}
$$
for all $(p,q)\in\widetilde{M}$. Note that if $f\in\Lip_0(M)$, then $\Phi f(p,q)=\duality{\mol{pq},f}$ and $\lipnorm{f}=\norm{\Phi f}_\infty$, so $\Phi$ is a linear isometry from $\Lip_0(M)$ into $C_b(\widetilde{M})$ --- the space of bounded continuous functions on $\widetilde{M}$, which can be identified with $C(\beta\widetilde{M})$, the space of real-valued continuous functions on the Stone-\v{C}ech compactification $\beta\widetilde{M}$ of $\widetilde{M}$. Its adjoint operator $\dual{\Phi}\colon\dual{C(\beta\widetilde{M})}\rightarrow\dual{\Lip_0(M)}$ is thus surjective; recall that $\dual{C(\beta\widetilde{M})}$ is just the space of Radon measures on $\beta\widetilde{M}$.

Now fix two distinct points $p,q$ in $M$ and consider $q$ to be the base point. Recall the following definition from \cite{AlGu_2019}:
\begin{align*}
\setd{pq}:=\big\{ \zeta\in\beta\widetilde{M}: \abs{\Phi f(\zeta)}=\lipnorm{f}\,\text{ whenever }f\in\Lip_0(M) &\text{ is such that}\\
&\Phi f(p,q)=\lipnorm{f}\big\}.
\end{align*}
Notice that $\setd{pq}$ is a compact subset of $\beta\widetilde{M}$ and that it always contains the points $(p,q)$ and $(q,p)$. In Proposition 3.5 in \cite{AlGu_2019}, the structure of the set $\setd{pq}$ was determined in the particular case when there is no $\xi\in\beta M$ such that $\slack{\xi;p,q}=0$ other than $p$ and $q$. Here, we generalize this result and show that, informally, $\setd{pq}$ lies inside $S\times S$, where $S=\set{\xi\in\beta M:\slack{\xi;p,q}=0}$ is the ``segment in the compactification''.

\begin{lemma}
\label{lm:dpq_full_version}
For any $\zeta\in\setd{pq}$ there is a net $(x_i,y_i)$ in $\widetilde{M}$ that converges to $\zeta$ in $\beta\widetilde{M}$, such that $\slack{x_i;p,q}$ and $\slack{y_i;p,q}$ converge to 0.
\end{lemma}

\begin{proof}
Let $\zeta\in\beta\widetilde{M}$, then there is a net $(x_i,y_i)$, $i\in I$, in $\widetilde{M}$ that converges to $\zeta$ in $\beta\widetilde{M}$, and we may choose a subnet such that $(x_i)$ and $(y_i)$ converge to elements $\xi$ and $\eta$, respectively, in $\beta M$; call this subnet $(x_i,y_i)$ again. We want to show that $\zeta\in\setd{pq}$ implies that $\slack{\xi;p,q}=\slack{\eta;p,q}=0$. To do so, we assume without loss of generality that $\slack{\eta;p,q}>0$ and we will construct $f\in\ball{\Lip_0(M)}$ such that $\Phi f(p,q)=1$ and $\abs{\Phi f(\zeta)}<1$, concluding thus that $\zeta\notin\setd{pq}$. There are three possibilities:

\begin{enumerate}[label={\upshape{(\roman*)}}]
\item $\slack{\xi;p,q}>0$.
\item $\slack{\xi;p,q}=0$ but $\xi\neq p,q$.
\item $\xi\in\set{p,q}$.
\end{enumerate}

Cases (i) and (iii) were dealt with in the proof of Proposition 3.5 in \cite{AlGu_2019} so we will only prove (ii), using a similar technique.

Suppose then that $\slack{\eta;p,q}>0$ and $\slack{\xi;p,q}=0$ but $\xi\neq p,q$. Since $p,q,\xi,\eta$ are all distinct, we may replace $(x_i,y_i)$ with a subnet such that the sets $\set{x_i:i\in I}$ and $\set{y_i:i\in I}$ are disjoint and do not contain $p$ or $q$. We now claim the following:

\begin{claim}
We may replace $(x_i,y_i)$, $i\in I$, with a subnet such that
$$
\inf_{i\in I} \frac{\slack{y_i;p,q}}{d(y_i,q)} >0 ,\; \inf_{i,j\in I} \frac{\slack{y_i;x_j,q}}{d(y_i,q)} >0,\; \inf_{i\in I} \frac{d(y_i,q)}{d(y_i,p)}>0 \quad \text{and} \quad \inf_{i,j\in I} \frac{d(y_i,q)}{d(y_i,x_j)} >0 \text{.}
$$
\end{claim}

\begin{proof}[Proof of the claim]
Since $\lim_i\slack{x_i;p,q}=0$, we may choose a subnet such that $d(x_i,q)$ is bounded. We may also either choose a subnet such that $d(y_i,q)\rightarrow\infty$ or one such that $d(y_i,q)$ is bounded. We split the proof into these two cases.

Suppose first that we take a subnet such that $d(y_i,q)\leq C_1$ for some $C_1<\infty$ and all $i\in I$. It is easy to check that the identity
$$
\slack{y;x,q}=\slack{x;y,p}+\slack{y;p,q}-\slack{x;p,q}
$$
holds for any $p,q,x,y\in M$. In particular, it implies that
$$
\slack{y_i;x_j,q}\geq\slack{y_i;p,q}-\slack{x_j;p,q}
$$
for any $i,j\in I$. Since $\lim_i\slack{x_i;p,q}=0$ and $\lim_i\slack{y_i;p,q}>0$, we may choose a subnet such that $\slack{y_i;p,q}\geq\delta$ and $\slack{x_i;p,q}\leq\delta/2$ for some $\delta>0$, so that
$$
\frac{\slack{y_i;x_j,q}}{d(y_i,q)}\geq\frac{\delta}{2C_1} \quad \text{and} \quad \frac{\slack{y_i;p,q}}{d(y_i,q)}\geq\frac{\delta}{C_1}
$$
for all $i,j\in I$. Also $y_i\rightarrow\eta\neq q$, hence we may take a subnet such that $d(y_i,q)\geq C_2$ for some $C_2>0$ and all $i\in I$. If $C_3<\infty$ is such that $d(x_i,q)\leq C_3$ for all $i\in I$, we obtain
$$
\frac{d(y_i,q)}{d(y_i,x_j)}\geq\frac{d(y_i,q)}{d(y_i,q)+d(x_j,q)}\geq\frac{C_2}{C_1+C_3}
$$
for all $i,j\in I$, and similarly $d(y_i,q)/d(y_i,p)\geq C_2/(C_1+d(p,q))$.

Now assume that we take a subnet such that $d(y_i,q)\rightarrow\infty$ instead. Then also $d(y_i,p)\geq d(y_i,q)-d(p,q)\rightarrow\infty$ and
$$
\limsup_i \frac{d(y_i,p)}{d(y_i,q)}\leq\limsup_i \frac{d(y_i,q)+d(q,p)}{d(y_i,q)}=1+\limsup_i\frac{d(p,q)}{d(y_i,q)}=1.
$$
By symmetry in $p$ and $q$ we get $\lim_i d(y_i,p)/d(y_i,q)=1$. Hence
$$
\lim_i \frac{\slack{y_i;p,q}}{d(y_i,q)}=1+\lim_i\frac{d(y_i,p)-d(p,q)}{d(y_i,q)}=2,
$$
so we may take a subnet where $\slack{y_i;p,q}/d(y_i,q)$ and $d(y_i,q)/d(y_i,p)$ are bounded below by a positive constant. Also, since $d(x_i,q)$ is bounded, we may choose a further subnet such that $d(x_j,q)/d(y_i,q)\leq1/2$ for all $i,j$, and then
$$
\frac{\slack{y_i;x_j,q}}{d(y_i,q)}=1+\frac{d(y_i,x_j)-d(x_j,q)}{d(y_i,q)}\geq\frac{1}{2}+\frac{d(y_i,x_j)}{d(y_i,q)}\geq\frac{1}{2} \,.
$$
Finally,
$$
\frac{d(y_i,x_j)}{d(y_i,q)}\leq 1+\frac{d(x_j,q)}{d(y_i,q)}\leq\frac{3}{2}
$$
and so $d(y_i,q)/d(y_i,x_j)\geq 2/3$ for all $i,j$.
\end{proof}

Now we continue with the proof of Lemma \ref{lm:dpq_full_version}. Using the Claim, replace $(x_i,y_i)$ with a subnet and choose $c>0$ and $\delta>0$ such that
$$
c<\min\set{2,\inf_{i,j\in I} \frac{\slack{y_i;x_j,q}}{d(y_i,q)},\inf_{i\in I} \frac{\slack{y_i;p,q}}{d(y_i,q)}}
$$
and
$$
\delta<\min\set{1,\inf_{i,j\in I} \frac{d(y_i,q)}{d(y_i,x_j)}, \inf_{i\in I} \frac{d(y_i,q)}{d(y_i,p)}} \,.
$$Let $X=\set{x_i:i\in I}$, $Y=\set{y_i:i\in I}$ and $Z=\set{p,q}\cup X\cup Y$. Define $f\colon Z\rightarrow\mathbb{R}$ by
$$
f(z)=\begin{cases}
d(z,q) & \text{if } z\in Z\setminus Y \\
(1-c/2)\cdot d(z,q) & \text{if } z\in Y \text{.}
\end{cases}
$$
It is clear that $\Phi f(p,q)=1$, $\abs{\Phi f(x,y)}\leq 1$ for $x,y\in Z\setminus Y$ and $\abs{\Phi f(x,y)}\leq 1-c/2$ for $x,y\in Y$. Moreover, if $y\in Y$ then $\Phi f(y,q)=1-c/2$, and for any $x\in X\cup\set{p}$ we have
\begin{align*}
1+\Phi f(y,x) &= \frac{\slack{y;x,q}-c/2\cdot d(y,q)}{d(y,x)} \geq (c-c/2)\cdot\frac{d(y,q)}{d(y,x)} \geq \frac{c\delta}{2}, \\
1-\Phi f(y,x) &= \frac{\slack{x;y,q}+c/2\cdot d(y,q)}{d(y,x)} \geq \frac{c}{2}\cdot\frac{d(y,q)}{d(y,x)} \geq \frac{c\delta}{2},
\end{align*}
so $\abs{\Phi f(y,x)}\leq 1-c\delta/2$. We conclude that $\lipnorm{f}=1$. Now extend $f$ from $Z$ to $M$. Then $f\in\ball{\Lip_0(M)}$, $\Phi f(p,q)=1$, and $\abs{\Phi f(\zeta)}=\lim_i\abs{\Phi f(x_i,y_i)}\leq 1-c\delta/2<1$, hence $\zeta\notin\setd{pq}$.
\end{proof}

Let us now define the set $\sete{pq}$ as
\begin{align}
\label{def:set_Epq}
\sete{pq}:=\big\{ \phi\in\ball{\dual{\Lip_0(M)}}: \text{$\duality{f,\phi}=1$ for every $f\in\ball{\Lip_0(M)}$} & \\\nonumber
\text{such that $\Phi f(p,q)=1$} \big\} &.
\end{align}
Notice that $\sete{pq}$ is a \weaks-compact and convex subset of the ball of $\dual{\Lip_0(M)}$ and that it contains $\mol{pq}$. The importance of this set lies in the following observation, which goes back to \cite{deLeeuw_1961}: suppose that $\mol{pq}$ is a convex combination of some $m,m'\in\ball{\lipfree{M}}$, i.e.
$$
\mol{pq}=tm+(1-t)m'
$$
for some $t\in (0,1)$. If $f\in\ball{\Lip_0(M)}$ is such that $\Phi f(p,q)=1$, the inequalities
\begin{align*}
1=\duality{\mol{pq},f} &=t\duality{m,f}+(1-t)\duality{m',f} \\
&\leq t\norm{m}\lipnorm{f}+(1-t)\norm{m'}\lipnorm{f}\leq 1
\end{align*}
hold and so $\duality{m,f}=\duality{m',f}=1$. It follows that $m,m'\in\sete{pq}$. Hence, in order to show that $\mol{pq}$ is an extreme point of $\ball{\lipfree{M}}$, it suffices to show that \mbox{$\sete{pq}\cap\lipfree{M}=\set{\mol{pq}}$.} To this end, we start with a generalization of Lemma 3.3 in \cite{AlGu_2019}:

\begin{lemma}
\label{lm:measure_bound}
Let $K$ be a closed subset of $\beta\widetilde{M}$ such that $\setd{pq}\cap K=\emptyset$. Then there is a constant $C$, depending on $K$, such that
$$
\abs{\mu}(K)\leq C\cdot(\norm{\mu}-1)
$$
for any measure $\mu\in\dual{C(\beta\widetilde{M})}$ such that $\dual{\Phi}\mu\in\sete{pq}$.
\end{lemma}

Here, $\abs{\mu}\in\dual{C(\beta\widetilde{M})}$ denotes the total variation of $\mu$, as usual.

\begin{proof}
Let $\zeta\in K$. Then $\zeta\notin\setd{pq}$ so there is an $f\in\ball{\Lip_0(M)}$ such that $\Phi f(p,q)=1$ and $\abs{\Phi f(\zeta)}<1$ and, since $\Phi f$ is continuous, there are $c_\zeta\in (0,1)$ and an open neighborhood $V_\zeta$ of $\zeta$ such that $\abs{\Phi f(\zeta^\prime)}\leq c_\zeta$ for every $\zeta^\prime\in V_\zeta$. Moreover, if $\mu\in\dual{C(\beta\widetilde{M})}$ is a measure such that $\dual{\Phi}\mu\in\sete{pq}$, then we have
\begin{align*}
1=\duality{f,\dual{\Phi}\mu}=\int_{\beta\widetilde{M}}(\Phi f)\,d\mu&=\int_{V_\zeta}(\Phi f)\,d\mu+\int_{\beta\widetilde{M}\setminus V_\zeta}(\Phi f)\,d\mu \\
&\leq c_\zeta\abs{\mu}(V_\zeta)+\abs{\mu}(\beta\widetilde{M}\setminus V_\zeta) \\
&= \norm{\mu}-(1-c_\zeta)\abs{\mu}(V_\zeta)
\end{align*}
hence
$$
\abs{\mu}(V_\zeta)\leq\frac{\norm{\mu}-1}{1-c_\zeta} \,.
$$

Now, $\set{V_\zeta:\zeta\in K}$ is an open cover of the compact set $K$, so it admits a finite subcover $K\subset\bigcup_{j=1}^{n} V_{\zeta_j}$. Thus, for any $\mu\in\dual{C(\beta\widetilde{M})}$ such that $\dual{\Phi}\mu\in\sete{pq}$ we have
$$
\abs{\mu}(K)\leq\sum_{j=1}^n\abs{\mu}(V_{\zeta_j})\leq C\cdot(\norm{\mu}-1),
$$
where $C=\sum_{j=1}^n (1-c_{\zeta_j})^{-1}<\infty$.
\end{proof}

The following lemma shows that if $\mol{pq}$ is a convex combination of two elements $m,m'$ of the unit ball, then $m,m'$ must be supported on the segment $[p,q]$.

\begin{lemma}
\label{lm:convex_comb_mol}
For the set $\sete{pq}$ defined as in (\ref{def:set_Epq}), we have $\sete{pq}\cap\lipfree{M}\subset\lipfree{[p,q]}.$
\end{lemma}

\begin{proof}
Let $\pi_1,\pi_2\colon\widetilde{M}\rightarrow M$ be the projection mappings given by $\pi_1(x,y)=x$ and $\pi_2(x,y)=y$. For a set $A\subset\widetilde{M}$, denote $\pi(A)=\pi_1(A)\cup\pi_2(A)$, i.e. $\pi(A)$ is the set of points of $M$ appearing as either coordinate of an element of $A$.

\begin{claim}
\label{lm:m_supp_constraint}
If $U$ is an open subset of $\beta\widetilde{M}$ such that $\setd{pq}\subset U$, then
$\sete{pq}\cap\lipfree{M}$ is a subset of $\lipfree{\pi(U\cap\widetilde{M})}$.
\end{claim}

\begin{proof}[Proof of the Claim]
Denote $N=\pi(U\cap\widetilde{M})$, and let $m\in\sete{pq}\cap\lipfree{M}$ and $k\in\mathbb{N}$. By Lemma \ref{lm:l1_series_argument}, $m$ admits a representation $m=\sum_{n=1}^\infty a_n\mol{p_nq_n}$ where $(p_n,q_n)\in\widetilde{M}$ for all $n\in\mathbb{N}$ and $\sum_{n=1}^\infty\abs{a_n}\leq 1+1/k$. Let $I=\set{n\in\mathbb{N}:(p_n,q_n)\in U}$ and $m_k=\sum_{n\in I} a_n\mol{p_nq_n}$. Notice that $p_n,q_n\in N$ for each $n\in I$, hence $m_k\in\lipfree{N}$.

It is not difficult to observe that $\dual{\Phi}\delta_{(x,y)}=\mol{xy}$ for any $(x,y)\in\widetilde{M}$, where $\delta_{(x,y)}\in\ball{\dual{C(\beta\widetilde{M})}}$ is the evaluation functional at $(x,y)\in\widetilde{M}$. Hence, if we denote $\mu=\sum_{n=1}^\infty a_n\delta_{(p_n,q_n)}$, then $\mu\in\dual{C(\beta\widetilde{M})}$ as the series is absolutely convergent, $\norm{\mu}\leq\sum_{n=1}^\infty\abs{a_n}\leq 1+1/k$, and $\dual{\Phi}\mu=m$.

Denote $K=\beta\widetilde{M}\setminus U$ and let $C$ be the constant assigned to $K$ by Lemma \ref{lm:measure_bound}. For each $f\in\Lip_0(M)$ we have
\begin{align*}
\duality{m-m_k,f} &= \sum_{n\notin I}a_n\Phi f(p_n,q_n) = \sum_{\set{n\in\mathbb{N}:(p_n,q_n)\in K}}a_n\Phi f(p_n,q_n) \\
&= \int_{\beta\widetilde{M}} (\Phi f)\cdot\chi_K \,d\mu = \int_K (\Phi f)\,d\mu,
\end{align*}
where $\chi_K$ is the characteristic function of $K$. So,
$$
\abs{\duality{m-m_k,f}} = \abs{\int_K (\Phi f)\,d\mu} \leq \norm{\Phi f}_\infty\cdot\abs{\mu}(K) \leq \lipnorm{f}\cdot C/k
$$
and $\norm{m-m_k}\leq C/k$. Hence, if $k\rightarrow\infty$ then $m_k\rightarrow m$ and thus $m$ is in the closed subspace $\lipfree{N}$ of $\lipfree{M}$.
\end{proof}

Now, to proceed with the proof of Lemma \ref{lm:convex_comb_mol}, take the continuous function $\varphi\colon\widetilde{M}\rightarrow [0,\infty)$ defined by
$$
\varphi(x,y)=\max\set{\slack{x;p,q},\slack{y;p,q}}
$$
and extend $\varphi$ to a continuous function $\varphi\colon\beta\widetilde{M}\rightarrow [0,\infty]$. Consider the sets
$$
S_n = \set{x\in M : \slack{x;p,q}\leq\tfrac{1}{n}}
$$
and
$$
U_n = \set{\zeta\in\beta\widetilde{M} : \varphi(\zeta)<\tfrac{1}{n}}
$$
for $n\in\mathbb{N}$. Notice that $U_n$ is open in $\beta\widetilde{M}$ and that $\pi(U_n\cap\widetilde{M})\subset S_n$ by definition. For each $n$, Lemma \ref{lm:dpq_full_version} implies that $\setd{pq}\subset U_n$, and applying the Claim we get $\sete{pq}\cap\lipfree{M}\subset\lipfree{S_n}$. Thus
$$
\sete{pq}\cap\lipfree{M}\subset\bigcap_{n=1}^\infty\lipfree{S_n} \,.
$$
Since all $S_n$ are subsets of $S_1$, we have also $\bigcap_{n=1}^\infty\lipfree{S_n}\subset\lipfree{S_1}$. We may therefore apply Theorem \ref{tm:f_intersection} to the bounded metric space $S_1$ and obtain
$$
\bigcap_{n=1}^\infty\lipfree{S_n}=\mathcal{F}\pare{\bigcap_{n=1}^\infty S_n}=\lipfree{[p,q]} \,.
$$
\end{proof}

The main result is now an immediate consequence of Lemma \ref{lm:convex_comb_mol}:

\begin{proof}[Proof of Theorem \ref{tm:main_theorem}]
First, let $\mu$ be a finitely supported extreme point of $\ball{\lipfree{M}}$ and let $K\subset M$ be its support. Then $\mu$ is also an extreme point of $\ball{\lipfree{K}}$, hence preserved extreme in $\ball{\lipfree{K}}$, and therefore it must be an elementary molecule $\mol{pq}$ by Corollary 2.5.4 in \cite{Weaver}. The fact that $[p,q]=\set{p,q}$ is proven easily, e.g. in Proposition 2.2 in \cite{AlGu_2019}. 

On the other hand, assume that $[p,q]=\set{p,q}$ and suppose that we have \mbox{$\mol{pq}=\frac{1}{2}(m+m')$} for some $m,m'\in\ball{\lipfree{M}}$. As we have already remarked before Lemma \ref{lm:measure_bound}, $m, m'$ must belong to $\sete{pq}$. Taking $q$ as the base point of $M$, Lemma \ref{lm:convex_comb_mol} implies that $m\in\lipfree{[p,q]}=\lipfree{\set{p,q}}=\lspan\embd(p)$. Since $\norm{m}=1$, it follows that \mbox{$m=\pm\embd(p)/d(p,q)=\pm\mol{pq}$.} The case $m=-\mol{pq}$ is clearly impossible, thus $m=\mol{pq}$ and $\mol{pq}\in\ext\ball{\lipfree{M}}$.
\end{proof}

\begin{example}
\label{ex:all_ext_unprev}
An application of Theorem \ref{tm:main_theorem} allows us to show that there exists a complete metric space $M$ such that all of its elementary molecules are extreme but none of them are preserved. Indeed, let $M$ be the space described in Example 2.4 in \cite{IvKaWe_2007}. It is shown there to have the following properties:

\begin{enumerate}[label={\upshape{(\roman*)}}]
\item it is a closed subset of a strictly convex Banach space,
\item it contains no nontrivial linear segments,
\item it is ``almost metrically convex'', i.e. a length space.
\end{enumerate}

It follows from (i) and (ii) that $M$ contains no nontrivial metric segments, hence Theorem \ref{tm:main_theorem} implies that all elementary molecules are extreme points of $\ball{\lipfree{M}}$. However, by (iii) and Proposition 5.9 in \cite{GaPrRu_2018}, $\ball{\lipfree{M}}$ has no preserved extreme point.
\end{example}

\section{Open questions}

The main question regarding extremal structure of Lipschitz-free spaces remains open and reads as follows.

\begin{question}[\cite{AlGu_2019}]
\label{q:ext}
Is every extreme point of $\ball{\lipfree{M}}$ an elementary molecule?
\end{question}

By the argument used in the proof of Theorem \ref{tm:main_theorem}, this is equivalent to all extreme points of $\ball{\lipfree{M}}$ having finite support. This is known to be true in certain cases where $M$ is bounded and $\lipfree{M}$ is the dual of either the well-known space $\lip_0(M)$ of ``little Lipschitz functions'' (see Section 3.3 of \cite{Weaver}) or a subspace thereof \cite{GPPR_2018}. This holds, in particular, when $M$ is compact and either countable \cite{Dalet_2015} or ultrametric \cite{Dalet_2015_2}. However, it is not known whether the answer to Question \ref{q:ext} is positive under the assumption that $M$ is compact, or that $\lipfree{M}$ is a dual Banach space, or even both at the same time. Maybe the results of Section \ref{section: intersections} of the present note could be helpful for studying the supports of extreme points of $\ball{\lipfree{M}}$ for bounded $M$.

Recall that $M$ is \emph{geodesic} if every pair of points $p,q\in M$ may be joined by an isometric copy of $[0,d(p,q)]\subset\mathbb{R}$, and that for complete $M$ this is equivalent to $[p,q]\neq\set{p,q}$ for any pair of different points $p,q\in M$ (see Proposition 4.1 in \cite{GaPrRu_2018}). Thus, Theorem \ref{tm:main_theorem} implies that $M$ is geodesic whenever $\ball{\lipfree{M}}$ has no extreme points. A positive answer to Question \ref{q:ext} would show that the opposite implication is also true.

Other extremal objects in $\ball{\lipfree{M}}$ which remain uncharacterized at the time of this writing are exposed points. In particular, it is not known whether they must be elementary molecules. It is shown in \cite{GPPR_2018} that all extreme points of $\ball{\lipfree{M}}$ are exposed under various circumstances, all of which involve $\lipfree{M}$ being a dual space. Extreme molecules are also automatically exposed (in fact, strongly exposed) in the cases \ref{case:e2_ultrametric} and \ref{case:e2_tree} listed in Proposition \ref{pr:state_of_the_art}. In view of all these partial results, it is natural to ask:

\begin{question}
\label{q:exp}
Are all extreme points of $\ball{\lipfree{M}}$ exposed?
\end{question}

\begin{remark*}
After completion of the present manuscript, we have learned that Petitjean and Proch\'azka \cite{PePr_2018}, and independently Garc\'ia-Lirola \cite{garcia-lirola_pc}, have found proofs by which from Theorem \ref{tm:f_intersection} it follows that every elementary molecule defined by points forming a trivial metric segment is in fact exposed. Thus, a positive answer to Question \ref{q:ext} would also imply a positive answer to Question \ref{q:exp}.
\end{remark*}

Finally, we would like to know whether the existence of supports, or equivalently the intersection property proved in Theorem \ref{tm:f_intersection}, holds for a more general class than bounded metric spaces. We are not aware of any counterexample.

\begin{question}
\label{q:supp}
Do the properties from Proposition \ref{prop: intersections_iff_supports} hold for any complete metric space $M$?
\end{question}

\section*{Acknowledgements}

This work was supported by the grant GA\v CR 18-00960Y. The first author was also partially supported by the Spanish Ministry of Economy and Competitiveness under Grant \mbox{MTM2014-57838-C2-1-P}.

This research was conducted while the first author visited the Faculty of Information Technology at the Czech Technical University in Prague in 2018. He is grateful for the hospitality and excellent working conditions during his stay.

The authors are grateful to L. C. Garc\'ia-Lirola, A. Proch\'azka, A. Rueda Zoca and to the anonymous referee for many valuable comments and remarks leading to improvement of this document. Example \ref{ex:all_ext_unprev} was suggested by A. Rueda Zoca.


\end{document}